\documentclass{article}

\usepackage{a4}
\usepackage{epsfig}

\usepackage{amsfonts}

\newtheorem{theo}{Theorem}[section]
\newtheorem{prop}[theo]{Proposition}
\newtheorem{corol}[theo]{Corollary}
\newtheorem{lem}[theo]{Lemma}

\newtheorem{defin}[theo]{Definition}

\newenvironment{proof}{{\flushleft \em Proof : }}{\hfill $\square$ \vspace{5mm}}

\newcommand{\N}{\mathbb N}
\newcommand{\R}{\mathbb R}

\begin{document}

\title{Generic measures for hyperbolic flows on non compact spaces}

\author{Yves Coudene, Barbara Schapira}

\date{}

\maketitle

\parindent=0pt

\begin{center}
\emph{IRMAR, Universit\'e Rennes 1, campus Beaulieu, b\^at.23
35042 Rennes cedex, France}\\
\emph{LAMFA, Universit\'e Picardie Jules Verne, 33 rue St Leu, 80000 Amiens, France}
\end{center}

\begin{abstract}
We consider the geodesic flow on a complete connected 
negatively curved manifold.
We show that the set of invariant borel probability measures 
contains a dense $G_\delta$-subset consisting of 
ergodic measures fully supported on the non-wandering set.
We also treat the case of non-positively curved manifolds
and provide general tools to deal with hyperbolic systems
defined on non-compact spaces. 
\footnote{37B10, 37D40, 34C28}{}
\end{abstract}



\section{Introduction}

In the context of hyperbolic dynamics, Axiom A flows admit many
ergodic invariant probability measures supported on 
each connected component $X$ of  the non-wandering set : 
the measure of maximal entropy and the SRB measure
are two examples which provide many informations on the asymptotic 
behaviour of trajectories. K. Sigmund \cite{si} 
showed even more: the ergodic probability measures with full support 
on $X$ form a $G_\delta$ dense subset of the set ${\cal M}^1(X)$ 
of all probability invariant measures. 

\medskip

\quad
Our goal is to extend these results to hyperbolic systems 
defined on non-compact spaces. We build our study on two remarkable 
properties of hyperbolic systems :
the existence of a product structure and the closing lemma.
Neither of these two properties relies on the compactness of 
the ambient space. Yet they are sufficient to obtain a general
density theorem, which shows that ergodic probability 
measures of full support are indeed abundant.

\medskip

\quad
The most famous example of hyperbolic system 
with noncompact phase space 
is given by the geodesic flow on a complete 
connected negatively curved manifold.
Our general density theorem implies the following result :

\begin{theo}\label{mainresult} 
Let $M$ be a negatively curved, connected, complete riemannian manifold. 
We assume that the non wandering set $\Omega$ of 
the geodesic flow is non empty.  
Then the $G_{\delta}$ subset of 
ergodic measures fully supported on $\Omega$ 
is dense in the set of all probability borel
measures on $T^1M$ invariant by the geodesic flow. 
\end{theo}

As the construction ultimately relies on the Baire category
theorem, the measures obtained are not explicit.
This result was previously known for geodesic flows on 
geometrically finite negatively curved manifolds. 
In that setting,  
there is an explicit construction of an
ergodic invariant probability Gibbs measure \cite{cou2}.

\medskip

\quad 
On the other hand, the theorem does not hold 
in the context of non-positive curvature.
We give a counterexample, and proceed to show that 
ergodic probability measures of full support are dense, 
if we assume that the non-wandering set is everything
and there is a hyperbolic closed geodesic.

\medskip

\quad 
Other examples of hyperbolic systems on non-compact manifolds are 
obtained by lifting an Anosov or Axiom A flow to some non-compact cover of 
the phase space. Our density result applies verbatim and gives 
the density of ergodic measures of full support, if we restrict
the system to a connected component of the non-wandering set on the cover.

\medskip
\quad
The set of all borel probability measures invariant by a dynamical
system will be denoted by ${\cal M}^1(X)$ in the sequel. 
It is endowed with the weak topology: a sequence $\mu_n$ of probability measures
converges to some measure $\mu$ if $\int f d\mu_n\rightarrow \int f d\mu$
for all bounded continuous $f$.
We refer to the classical treatise 
of Billingsley \cite{bi} for a detailed study of that topology.
In particular, we will use the fact that 
we may restrict ourselves to bounded \emph{Lipschitz} functions
in the definition of weak convergence.

\medskip
\quad
In the first two sections, using first the closing lemma, and then the local
product structure of a hyperbolic flow, we prove the density of the Dirac
measures supported on periodic orbits in the set of all invariant probability
measures. In section 4 we prove theorem \ref{mainresult}. 
We then extend this result to rank one
manifolds, and to hyperbolic systems on covers.
We end by an appendix containing a proof of the closing lemma 
for rank one vectors on non-positively curved manifolds.


\section{The closing lemma and ergodic measures}

The geodesic flow on the unit tangent bundle of a negatively curved manifold 
satisfies a property known as a closing lemma.
Roughly speaking, a piece of orbit that comes back close to its 
initial point is in fact close to a periodic orbit.
Here is a more formal statement :

\begin{defin}\label{closing_lemma} A flow $\phi_t$ on a metric space $X$ satisfies the {\em closing
 lemma} if for all points $v\in X$, there exists a neighborhood
 $V$ of $v$ which satisfies :\\
 for all $\varepsilon>0$, there exists a $\delta>0$ and a $t_0>0$
 such that for all $x\in V$ and all $t>t_0$ with $d(x,\phi_t x)<\delta$ and 
 $\phi_t x\in V$,
 there exists $x_0$ and $l>0$, with  $|l-t|<\varepsilon$,
 $\phi_{l}x_0=x_0$, and $d(\phi_s x_0,\phi_s x)<\varepsilon$
 for $0<s<\min(t,l)$.
\end{defin}

Note that for such flow, the non-wandering set and
the closure of the recurrent points coincide
with the closure of the set of  periodic points. 

\medskip

The first closing lemma was proven by Hadamard in 1898,
for negatively curved surfaces embedded in ambient space.
D. V. Anosov generalized the closing lemma to the systems that now
bear his name \cite{an} ; for this reason, it is often called the 
Anosov closing lemma. A proof of the closing lemma valid for
rank one manifolds and higher rank locally symmetric spaces 
can be found in \cite{eb1}. 
We recall its argument at the
end of the paper.

\begin{lem}\label{density_periodic_orbits_in_erg_probas} 
Let  $X$ be a metric space, $\phi_t$ a flow on $X$ 
which satisfies the closing lemma.
Then the set of Dirac probability measures 
supported on periodic orbits of the flow
is dense in the set of all invariant 
ergodic borel probability measures $\mu$ on~$X$.
\end{lem}

\begin{proof}
Let $f_i$ be finitely many bounded Lipschitz functions on $X$,
with Lipschitz constants bounded by $K$.
From the Poincar\'e recurrence theorem and the 
Birkhoff ergodic theorem, we know that for almost all $x\in X$,
$x$ is recurrent and the Birkhoff averages 
of $f_i$ evaluated at $x$ 
converges to $\int f_i\,d\mu$. Fix $\varepsilon>0$ and choose 
$t_0>0$ large enough so that  for $t\ge t_0$, 
$d(\phi_t x,x)<\delta$ and for all $i$, 
$\left|\frac{1}{t}\int_0^t f_i\circ \phi_s(x)ds-\int_X f_i\,d\mu\right|\le \varepsilon$. 
Let $x_0$ be the periodic point of period $l$ given by the closing lemma. 
We obtain
$$
\vbox{
\halign{$#$ &#\  $\leq$\ & $#$\hfill\cr
\left| \int f_i\,d\mu-{1\over l}\int_0^l f_i(\phi_sx_0)\,ds \right| 
&&
\left| {1\over t}\int_0^t f_i(\phi_s(x))\,ds
- {1\over l}\int_0^l f_i(\phi_sx_0)\,ds \right|+\varepsilon \cr 
&&
{1\over t}\int_0^t K\, d(\phi_sx,\phi_sx_0)\,ds+2 |t-l|\,||f_i||_\infty
+\varepsilon \cr 
&&
\vphantom{\Big|}(K+2||f_i||_\infty+1)\,\varepsilon\cr}
}
$$
This shows that the Dirac measure on that periodic 
orbit is close to the measure $\mu$.
\end{proof}

\quad If the space $X$ is {\em borel standard}, 
that is, a borel subset of a separable complete metric space,
the decomposition theorem of Choquet 
implies that the convex set generated by the invariant 
ergodic borel probability measures is dense in the set of all invariant
probability measures. We have proven:



\begin{corol}\label{density_convex_periodic_in_proba_measures} 
Assume moreover that $X$ is borel standard.
Then the convex subset 
generated by the normalized Dirac measures 
on periodic orbits is dense in 
the set of all invariant borel probability measures ${\cal M}^1(X)$.
\end{corol}

There is another path that leads to the previous result:
first prove a Lifsic type theorem \cite{li}; that is, 
all functions whose integrals
on periodic orbit vanish are coboundaries. For these functions,
deduce that the integrals with respect to all invariant probability 
measures is zero; this requires
additional work since the function given by the coboundary equation 
cannot be expected to be integrable when the underlying space is not compact.
Then end the argument by applying a Hahn-Banach theorem adapted to the 
 weak topology on the space of measures.

\section{Product structure and Dirac measures}

Our next goal is to show that Dirac measures on periodic orbits
are in fact dense in the set of all invariant measures.
This will require  an additional assumption.
 Let $X$ be a metric space, $\phi_t$ a continuous flow on $X$. 
 The strong stable sets of the flow are defined by~:

\smallskip

 $W^{ss}(x):=\lbrace \  y\in X \ | \ 
 \lim_{t\rightarrow \infty}
 d(\phi_{t}(x),\phi_{t}(y)) = 0 \ \  \rbrace$ ; 
 \\
 $W^{ss}_{\varepsilon}(x):=\lbrace \  y\in W^{ss}(x) \ | \ 
 d(\phi_{t}(x),\phi_{t}(y))\leq \varepsilon$ for all $t\ge 0 
 \ \  \rbrace$.

\smallskip

  One also defines the strong unstable sets 
 $W^{su}$ and $W^{su}_{\varepsilon}$ of 
 $\phi_t$ ; these are the stable sets of 
 $\phi_{-t}$.


\begin{defin}\label{local_product_structure} \
The flow ${\phi}_t$ is said to admit a 
{\em local product structure} if all points $v\in X$ have   
a neighborhood $V$ which satisfies : for all
$\varepsilon>0$, there exists a positive constant $\delta$,
such that for all $x,y\in V$ with  $d(x,y)\leq \delta$, 
there is a point $<x,y>\in X$, 
a real number $t$ with $|t|\leq \varepsilon$, so that: 
$$<x,y>\ \in\ {W}^{su}_{\varepsilon}\bigl({\phi}_{t}(x)\bigr)\cap 
{W}^{ss}_{\varepsilon}(y).$$
\end{defin}

This property is summarized by the 
following picture~:

\medskip




\begin{center}
\epsfig{figure=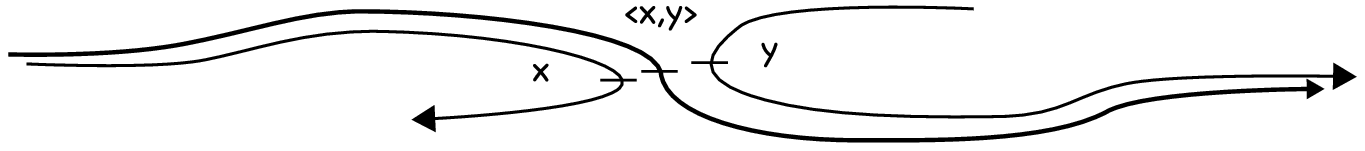,width=0.8\textwidth,angle=0}
\end{center}

\medskip

This assumption allows us to ``glue'' orbits together.
Note that we can glue an arbitrary finite number of
pieces of orbits, just by using the product structure 
recursively. In the next picture,
we first glue the trajectory of $x_1$ with the trajectory of 
$x_2$. The resulting dotted trajectory is then glued with $x_3$.
We end up with an orbit which follows first the orbit of $x_1$,
then the orbit of $x_2$ and finally the orbit of~$x_3$.




\begin{center}
\epsfig{figure=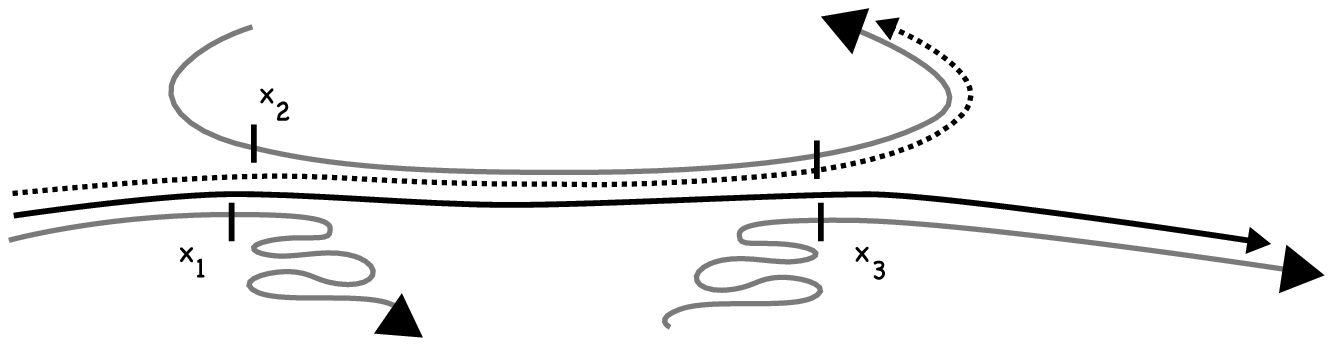,width=0.7\textwidth,angle=0}
\end{center}

\medskip

\quad
Geodesic flows on negatively curved manifolds admit a 
local product structure. In fact, there is even a global 
product structure on the unit tangent bundle of the
universal cover of the manifold. 

\medskip

\quad We are now ready to prove :

\begin{prop} Let  $X$ be a metric space, $\phi_t$ a transitive flow on $X$ 
admitting a local product structure and satisfying the closing lemma.
Then the set of normalized Dirac measures on periodic orbits is dense
in its convex closure.
\end{prop}

From this proposition and  corollary
\ref{density_convex_periodic_in_proba_measures},
we deduce that the 
set of normalized Dirac measures on periodic orbits is dense in the set 
of all invariant probability measures ${\cal M}^1(X)$. 


\begin{proof} 
We show that any convex combination of Dirac measures on periodic 
orbits can be approached by a Dirac measure on a single periodic 
orbit. So, let $x_1$, $x_3$,...,$x_{2n-1}$ be periodic points 
with period $l_1$, $l_3$,...,$l_{2n-1}$, and $c_1$,$c_3$,...,$c_{2n-1}$
positive real numbers with $\Sigma\, c_{2i+1}=1$.
Let us denote the Dirac measure on a periodic orbit of 
some point $x$ by $\delta_x$.
We want to find a periodic point $x$ such that 
$\delta_x$ is close to the sum $\Sigma\, c_{2i+1}\,\delta_{x_i}$.
The numbers $c_{2i+1}$ may be assumed to be rational numbers of the 
form $p_{2i+1}/q$.

\medskip

By transitivity, for each integer $i$ less than $n$, 
we can find a point $x_{2i}$ close to $x_{2i-1}$
whose trajectory becomes close to $x_{2i+1}$, say, after time 
$t_{2i}$. We can also find a point $x_{2n2}$ close to 
$x_{2n-1}$ whose trajectory becomes close to $x_1$ after some time.

\medskip




\begin{center}
\epsfig{figure=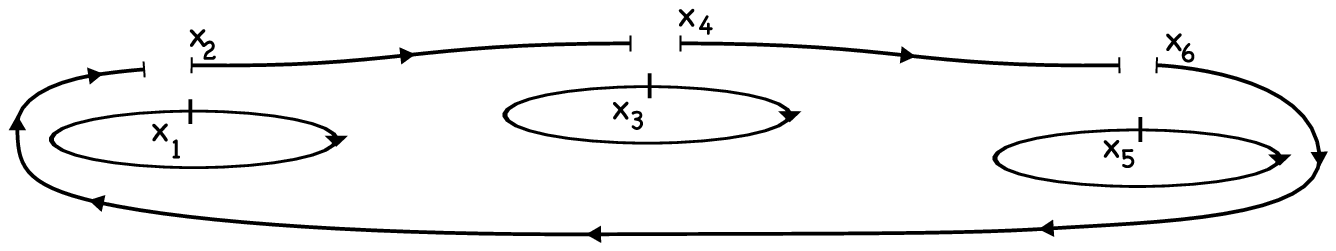,width=0.8\textwidth,angle=0}
\end{center}

\medskip

Now these trajectories are glued together. We fix an integer $N$,
that will be chosen big at the end. First glue the piece of orbit
starting from $x_1$, of length $Nl_1p_1$, together with 
the orbit of $x_2$, of length $t_2$. The resulting orbit ends 
in a neighborhood of $x_3$, and that neighborhood does not depend 
on the value of $N$. This orbit is glued with the trajectory 
starting from $x_3$, of length $Nl_2p_2$, and so on. The 
reader may want to write down the sequence of 
$\varepsilon$ obtained through the successive gluings. 
This is done in \cite{cou}, where a similar argument was used. 

\medskip

We end up with a trajectory starting close to $x_1$, doing $N$
turns around the first periodic orbit, then following the trajectory of
$x_2$ until it reaches $x_3$ ; 
then it turns $N$ times around the second periodic 
orbit, and so on, until it reaches $x_{2n}$ and goes back to $x_1$.

\medskip




\begin{center}
\epsfig{figure=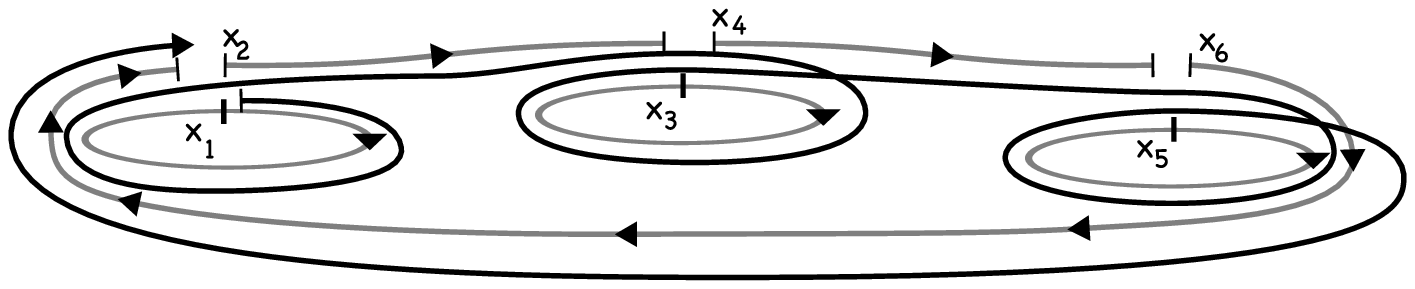,width=0.8\textwidth,angle=0}
\end{center}

\medskip
For sake of clarity, the trajectory only makes one turn around 
the periodic orbits on the pictures.
Finally, we use the closing lemma to obtain the periodic orbit we are
looking for.

\medskip




\begin{center}
\epsfig{figure=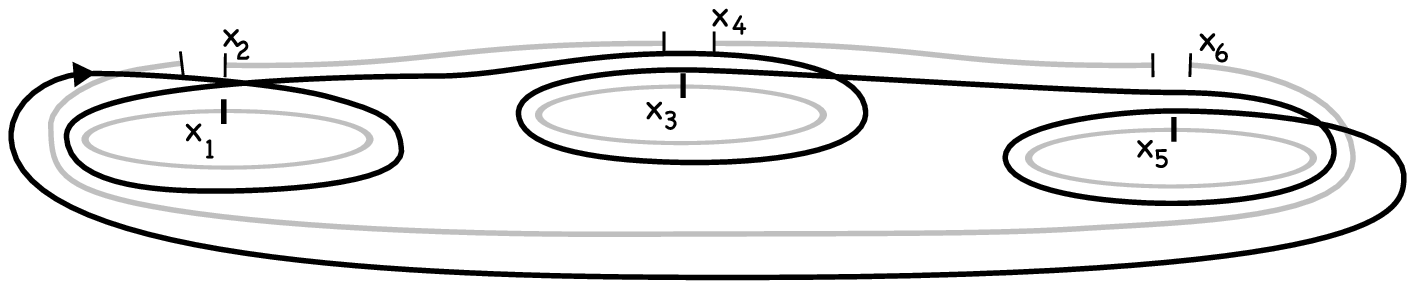,width=0.8\textwidth,angle=0}
\end{center}

\medskip

Integrating a function on that closed orbit gives a result close to
what we would have obtained if we integrate first $Np_1$ times around
the first closed orbit, then on the piece of orbit leading to $x_3$,
followed by the second periodic orbit, $Np_2$ times, and so on.
This amounts to integrating with respect to the measure:
$$
{\Sigma\ N\,p_{2i+1}\,\delta_{x_{2i+1}}\ + \ \Sigma\ \delta_{x_{2i}}\over
\Sigma\ N\,p_{2i+1}\ +\ \Sigma\ t_{2i}}
$$

If $N$ is big enough, the contribution of the pieces of orbits 
associated to the $x_{2i}$ is small, and this measure is close to
 $\Sigma\, {p_{2i+1}\over q}\,\delta_{x_{2i}}$; so we are done
\end{proof}

{\bf Remark :} \\
Under the hypothesis of the previous proposition, 
 transitivity is equivalent to the connectedness of $X$ together
with the density of the periodic orbits on $X$ \cite{cou}.


\section{Ergodic measures with full support}

We first prove that there are always probability measures with
full support, as soon as there are enough periodic orbits.
Recall that a \emph{polish space} is a topological space homeomorphic
to a separable complete metric space. Locally compact separable 
metric spaces are polish and the set of borel
probability measures on a polish space is again a polish space
with respect to the weak topology.
Hence, the Baire property holds in that context.

\begin{lem} Let $X$ be a polish space, 
$\phi_t$ a flow defined on $X$.
We assume that the periodic orbits of the flow form a dense 
subset in $X$. Then the set of borel probability measures of full
support is a dense $G_{\delta}$  subset of the set of all invariant
borel probability measures. 
\end{lem}

\begin{proof}
Let $\{U_i\}_{i\in{\bf N}}$ be a countable basis of open sets for the 
topology of $X$. 
Let $F_i$ be the closed set of  all  probability measures 
which are zero on $U_i$.
The complement of the union of all the $F_i$ is exactly the set 
of borel probability measures with full support.
Let us show that the $F_i$ have empty interiors.
Take $\nu\in F_i$ and $x_0$ a periodic point in $U_i$;
the Dirac measure on the orbit of $x_0$ is denoted by $\delta_{x_0}$.
Then the sequence $(1-{1\over n})\nu+{1\over n}\delta_{x_0}$ is not 
in $F_i$ and converges to $\nu$. The Baire theorem can now be 
applied and the lemma is proven.
\end{proof}

\quad
We can now prove a genericity result concerning 
ergodic invariant measures with full support.

\begin{theo}\label{generalresult} Let  $X$ be a polish space, 
$\phi_t$ a transitive continuous flow on $X$ 
admitting a local product structure and satisfying the closing lemma.
Then the set of invariant ergodic borel probability measures with full 
support is a dense $G_{\delta}$ subset of the set of all invariant
borel probability measures.
\end{theo}

Theorem \ref{mainresult} is a particular case of this result. 

\begin{proof} As noted above, the periodic orbits are dense in $X$.
From the previous lemma, we know that the set of probability measures 
with full support is a  dense $G_{\delta}$ subset of 
the set of all invariant borel probability measures ${\cal M}^1(X)$.
The proposition shows that the set of ergodic invariant probability measures 
is also dense in ${\cal M}^1(X)$. It remains to check that 
this set is a $G_{\delta}$ set.

\medskip

Let $\{f_i\}_{i\in \bf N}$ be some countable algebra 
of Lipschitz bounded functions which sepa\-rate points on $X$. 
This set is dense in $L^1(m)$, for all borel probability measure $m$~; 
this fact goes back to Rohlin; we refer to \cite{cou3} for a short proof.
We can know write the complement of the set of ergodic probability measures as 
the union of the following closed sets: 
$$
\vbox{\halign{$#$ &\hfill $#$\cr
  F_{k,l,i}=
& \{m \in {\cal M}^1(X)\ |\ \exists\, \mu,\nu\in\ {\cal M}^1(X),
  \alpha\in \big[{1\over k},1-{1\over k}\big],\ s.t.\ 
  \ m=\alpha\mu+(1-\alpha)\nu\ \cr
& and \   \int f_id\mu\,\geq \int f_i d\nu+{1\over l}\ \}\cr
}}
$$
The condition $\int f_id\mu\,\geq \int f_i d\nu+{1\over l}$ insures that the 
two measures $\mu$ and $\nu$ are not equal. This is a closed condition.
Now let  $m_n$ be a sequence of measures that can be written as 
barycenters $\alpha_n\mu_n+(1-\alpha_n)\nu_n$. If the sequence $m_n$
is converging in ${\cal M}^1(X)$, then it is tight. 
Up to a subsequence, we can assume that $\mu_n$ and $\nu_n$ 
converge to probability measures $\mu$ and 
$\nu$ of mass $\le 1$, and 
$\alpha_n\to \alpha\in [\frac{1}{k},1-\frac{1}{k}]$. 
The relation $m=\alpha\mu+(1-\alpha)\nu$ 
implies that the limit $m$ of $m_n$ is
a barycenter, and $F_{k,l,i}$ is closed.
This ends the proof.
\end{proof}


\section{Non-positively curved manifolds}

We have just proven that a geodesic flow on a negatively curved 
manifold always admits an ergodic invariant measure fully supported
on the non-wandering set of the flow.
We now give generalisations of this result to non-positively curved 
manifolds.

\subsection{Closing lemma and local product}
\quad
Let $M$ be a riemannian manifold and $v$ a vector belonging to the 
unit tangent bundle $T^1M$ of $M$. The vector $v$ is a 
\emph{rank one vector},
if the only parallel Jacobi fields along the geodesic generated by $v$
are proportional to the generator of the geodesic flow.
A connected complete non-positively curved manifold is said to be a 
\emph{rank one manifold} 
if its tangent bundle admits a rank one vector. In that case, the set of 
rank one vectors  is an open subset of $T^1M$. 
Rank one vectors generating closed geodesics 
are precisely the 
hyperbolic periodic points of the geodesic flow.

\medskip

\quad
Negatively curved manifolds are rank one manifolds.
There are also examples of rank one manifolds 
for which the sectional curvatures vanish on an open subset
of the manifold.
The concept of rank one manifold was introduced by 
W. Ballmann, M. Brin, P. Eberlein \cite{bbe};
we refer to the survey of G. Knieper \cite{kni} for an overview 
of the properties of rank one manifolds.

\bigskip

\quad
The goal is now to find a big subset of the non wandering set $\Omega$,
so that the geodesic flow $\phi_t$ 
has a local product structure and satisfies the closing lemma in restriction
to that subset.
Let us consider the set $\Omega_1\subset T^1M$ 
of non-wandering rank one vectors 
satisfying the following condition:

\smallskip

{\em
If $w\in T^1M$ is such that the distance $d(\phi_t(w),\phi_t(v))$
stays bounded for $t\geq 0$, then there is a $t_0\in {\bf R}$
such that  $\phi_{t_0}(w)\in W^{ss}(v)$.
}
\smallskip

In other words, any non-wandering 
half-geodesic staying at a bounded distance 
from a geodesic in $\Omega_1$ is in fact asymptotic to that geodesic.
Typical examples of vectors which are not in $\Omega_1$ are 
vectors whose associated geodesic is asymptotic to a closed geodesic 
belonging to a flat half-cylinder.

\begin{prop} 
Let $M$ be a rank one manifold.
Then the restriction of the geodesic flow to $\Omega_1$ 
admits a local product structure and satisfies the 
closing lemma. 
\end{prop}

\begin{proof}
The geodesic flow, in restriction to the set of non-wandering rank one 
vectors, satisfies the closing lemma.
This is proven by P. Eberlein in \cite{eb1}, prop 4.5.15,
in the general setting of non-positive curvature.
We recall the proof in an appendix at the end of the paper,
since the rank one hypothesis allows for some simplification.

We also know that  recurrent rank one vectors belong 
to $\Omega_1$ \cite{kni}, prop. 4.4.
Hence, rank one periodic orbits are in $\Omega_1$.
Now, if we apply the closing lemma to a non-wandering rank one 
vector, the closed orbit obtained is close to the rank one 
vector. Hence, as the set of rank one vectors is open, 
it is a closed rank one orbit and belongs
to $\Omega_1$. So the closing lemma is satisfied in restriction
to $\Omega_1$.

\quad
Let us consider two vectors $v,w$ in $T^1M$ which are close to each other.
Then there is a point $w'$ close to $v$ and $w$, such that 
$d(\phi_t(w),\phi_t(w'))$ is bounded for negative times, 
and $d(\phi_t(v),\phi_t(w'))$ is bounded for positive times
\cite{kni}, lemma 1.3.
So, if $v$ and $w$ belong to $\Omega_1$, then $w'$ is a rank one 
vector whose associated geodesic is positively asymptotic
trajectory of $v$, and negatively asymptotic to the trajectory
of $w$. This shows that $w'$ is in $\Omega_1$ and that there is 
a local product structure in restriction to $\Omega_1$. 
\end{proof}

\begin{prop} 
Let $M$ be a rank one manifold.
Then the restriction of the geodesic flow to $\Omega_1$ 
is transitive.
\end{prop}

\begin{proof}
Let $U_1$ and $U_2$ be two open sets containing points in $\Omega_1$.
We show that there is a trajectory in $\Omega_1$ that starts from
$U_1$ and ends in $U_2$. This will prove transitivity on $\Omega_1$.
From the previous proposition, we have seen that rank one 
periodic vectors are dense in $\Omega_1$.
So there exists rank one periodic vectors $v_1\in U_1$
and $v_2\in U_2$.
 There is also a rank one vector $v_3$
whose trajectory is negatively asymptotic to the trajectory of 
$v_1$ and positively asymptotic to
the trajectory of $v_2$, cf \cite{kni} lemma 1.4 ff.\\ 
Let us show that $v_3$ is non-wandering.
First note that there is also a trajectory  negatively asymptotic to 
$v_2$ and positively asymptotic to
the trajectory of $v_1$. That is, the two periodic orbits $v_1$, $v_2$
are heteroclinically related.




\begin{center}
\epsfig{figure=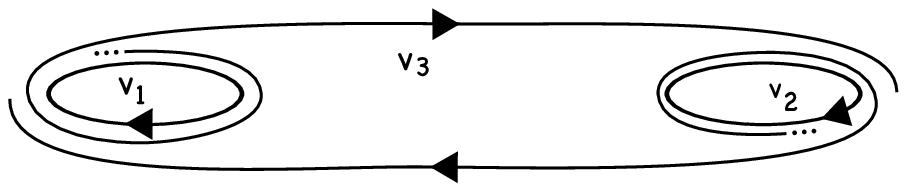,width=0.4\textwidth,angle=0}
\end{center}

This implies that 
the two connecting orbits are non-wandering : 
indeed, using the local product structure, we can glue
the two connecting orbits to obtain a trajectory that 
starts close to $v_3$, follows the second connecting orbit,
and then follows the orbit of $v_3$, coming back to the vector 
$v_3$ itself. Hence $v_3$ is in $\Omega$. Since it is 
asymptotic to $v_2$, it belongs to $\Omega_1$ and we are done.
\end{proof}

\medskip

Theorem \ref{generalresult} applied in restriction to $\Omega_1$ gives the: 

\begin{corol}
Let $M$ be a non-positively curved connected complete riemannian manifold.
Then there exists an ergodic finite borel measure on  $T^1M$, 
invariant with respect to the geodesic flow, whose support 
contains all hyperbolic periodic orbits.
\end{corol}

\subsection{Rank one manifolds}

 There are examples of rank one manifolds for which 
$\Omega_1$ is not dense in the non-wandering set $\Omega$.

\medskip



\begin{center}
\epsfig{figure=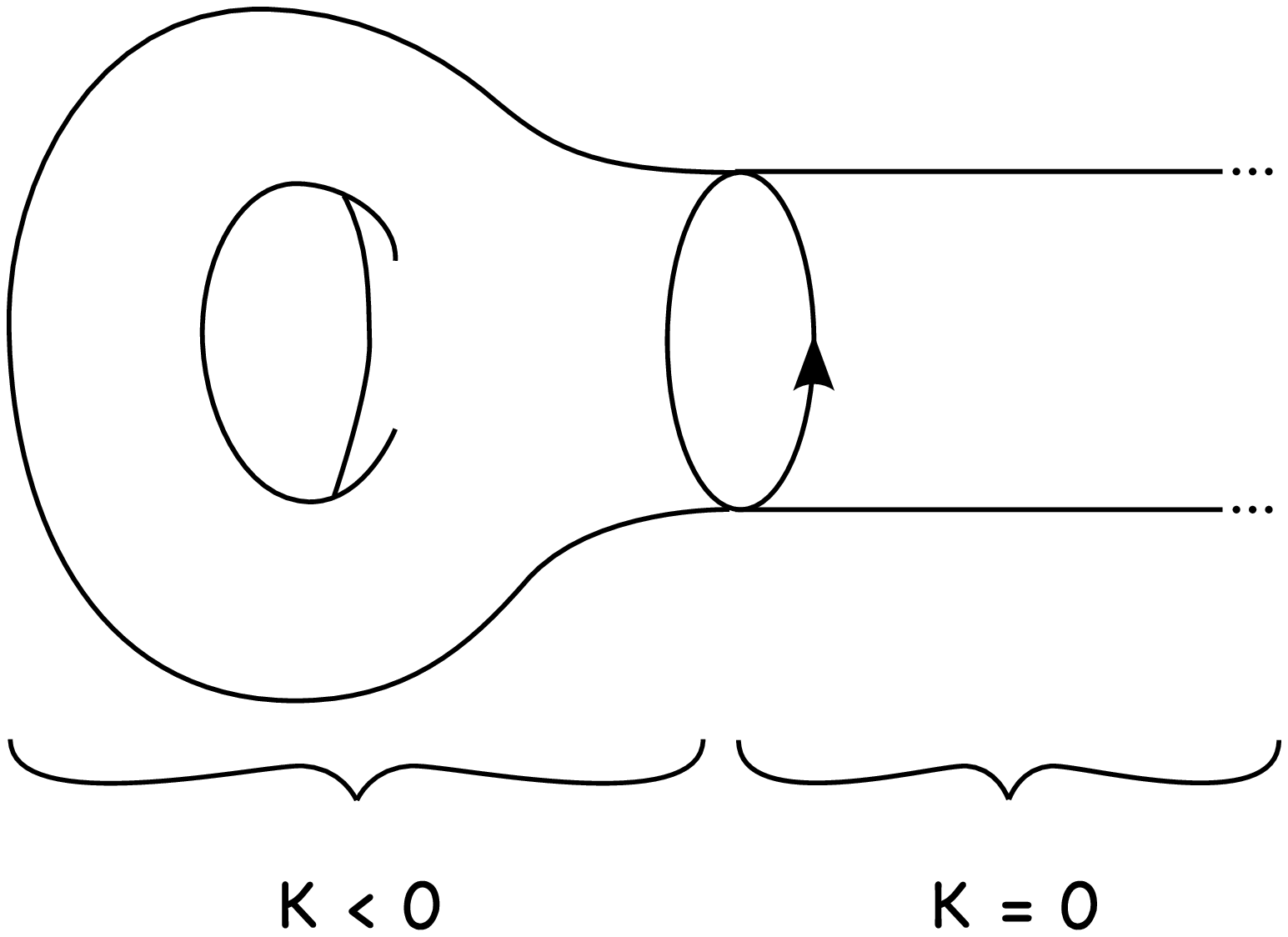,width=0.4\textwidth,angle=0}
\end{center}

\medskip

For example, let us consider a non-positively curved surface,
with a geodesic $\gamma$ bounding an half euclidean cylinder.
This corresponds to the right part on the picture above.
The remainder of the surface is compact, negatively curved.
This is depicted on the left hand side of the picture.
There are three types of geodesics on the surface :

\medskip

-- Closed geodesics in the cylinder part, which are parallel to $\gamma$.
   
\smallskip

-- Geodesics which cross $\gamma$. They are associated
   to rank one wandering vectors.

\smallskip

-- Geodesics that stay in the negatively curved part of the surface.
   They are associated to rank one 
   vectors.

\medskip

The non-wandering set $\Omega$ extends arbitrarily far on the right
whereas the set $\Omega_1$ is contained in the left part of the picture.
The only ergodic invariant probability measures giving positive measure to
the cylinder part are the Dirac measures on the periodic orbits.
This gives a rank one example with no ergodic invariant finite measure 
supported by $\Omega$.

\bigskip

\quad
There is however a simple condition that insures both the density of 
$\Omega_1$ in $\Omega$ and the transitivity on $\Omega_1$:
if $\Omega=T^1M$, then the rank one vectors form an open dense 
subset of $T^1M$; moreover we have seen that
 the geodesic flow is transitive;
this was first proven by Eberlein, see e.g.
 \cite{eb1}, prop. 4.7.3 and 4.7.4. 
This implies the existence of a rank one vector with dense orbit.
As mentioned earlier, such a vector belongs to $\Omega_1$; we refer to 
G. Knieper, \cite{kni} prop 4.4. 
We have established:

\begin{theo} Let $M$ be a rank one
manifold. We assume that 
the non-wandering set of the geodesic flow coincides with $T^1M$. 
Then there exists an ergodic borel probability 
measure on $T^1M$, which is invariant by the geodesic flow, and 
of full support.
\end{theo}

\section{Additional examples}

Theorem \ref{generalresult} can be applied to other examples:

\medskip

$\bullet$ \emph{Irreducible Markov chains defined on a countable alphabet.}

\smallskip

   We recover the density of Dirac measures on periodic points,
   in the set of all probability invariant measures, a result first 
   proven by Oxtoby \cite{Oxtoby} in the case of a finite alphabet.

\medskip
  
$\bullet$ \emph{Skew-products over Markov chains.} 

\smallskip

   The fiber has to be discrete, so that both the 
   product structure and the closing lemma lift to the    skew-product. The transitivity 
   on the skew-product was studied, for example, in \cite{cou}.
   When the fiber is a non-amenable group,
   the lift of a probability measure on the basis cannot be ergodic;
   this result is due to Zimmer \cite{zi}.   
   So the ergodic measures on the skew-product must be singular
   with respect to the action of the group extension.

\medskip

$\bullet$ \emph{Lift of hyperbolic systems to coverings of the base space.}

\smallskip

   For these systems, lifts of Gibbs measures are known to
   be ergodic only if the deck tranformation group is virtually
   $\bf Z$ or ${\bf Z}^2$ \cite{guivarch}.
   From this viewpoint, even the 
   existence of an infinite ergodic measure of full support 
   would have appeared somehow surprising. 
   Since these systems have received much attention 
   recently, we rephrase Theorem \ref{generalresult} in that context.

\begin{theo}
Let $X$ be a compact manifold, $\hat{X}$ a covering of $X$ and 
$\hat{\phi}_t$ a transitive flow on $\hat{X}$ which is the lift of an
Anosov flow defined on $X$. Then the set of invariant ergodic borel
probability measures with full support in $\hat{X}$  is a dense 
$G_\delta$-set in the set of all invariant borel probability measures.
\end{theo}

An analogous result holds for Axiom A systems, in restriction 
to preimages of basic pieces. 

\medskip

\quad
Finally, we remark that our results are again true 
if we look at transformations instead of flows. The proofs apply verbatim.


\section{Appendix : the rank one closing lemma}

We give a proof of the closing lemma for 
the geodesic flow on a non-positive complete manifold,
in restriction to the rank one non-wandering vectors.
There are two motivations for providing a proof :
first the rank one hypothesis allows for some simplification,
as compared with the classical proof given by P. Eberlein
\cite{eb1}.
Second, in the context of non-positive curvature, 
it seems that it is still unknown whether the closing lemma
holds for all non-wandering vectors.

\begin{theo}[Eberlein \cite{eb1}] Let $\tilde{M}$ be a Hadamard rank one manifold, and $M=\tilde{M}/\Gamma$,
with $\Gamma$ a discrete subgroup of $Isom(\tilde{M})$. 
Let $C\subset T^1 M$ be a compact set of rank one vectors and $\varepsilon>0$. 
Then there exists $\delta>0$ and $T>0$ with the following property:\\
If there exists $t\ge T$ and $v\in C$ such that $D(g^t v,v)<\delta$, then
there exists $t'\in ]t-\varepsilon,t+\varepsilon[$ and $v'$ a rank one vector
    such that  $g^{t'}v'=v'$ and $D(g^s v,g^s v')\le\varepsilon$ for all $0\le
    s\le \min(t,t')$. 
\end{theo}

 Here, we denote by $D$ the
natural distance on $T^1 M$ (and $T^1\tilde{M}$) 
induced by the riemannian metric on $M$.

\begin{proof}
Assume by contradiction that there exists a compact set 
$C\subset T^1 M$  of rank one vectors, 
$\varepsilon>0$, and a sequence $(v_n)_{n\in\N}$ of vectors of $C$, 
together with a sequence $t_n\rightarrow \infty$,
such that $D(v_n,g^{t_n}v_n)\to 0$, without periodic orbits $w_n$ of 
period $\omega_n$ $\varepsilon$-close to $t_n$ staying in their $\varepsilon$-neighbourhood
during the time $t_n$. 
By compactness, we can assume that $v_n$ converges to 
some $v\in C$, as $n$ goes to infinity. 

\quad

Lift the situation to $\tilde{M}$. There exists a compact set
$\tilde{C}\subset T^1\tilde{M}$,
$\varepsilon>0$, a sequence 
$(\tilde{v}_n)_{n\in\N}\in (\tilde{C})^{\N}$ of rank
one vectors converging to $\tilde{v}$,   
$t_n\ge n$ and $\varphi_n\in \Gamma$, such that
$D(\tilde{v}_n,d\varphi_n\circ g^{t_n}\tilde{v}_n)\le \frac{1}{n}$.

\quad 
Moreover, there exists no vector $\tilde{w}_n$ such that 
$D(\tilde{w}_n,\tilde{v}_n)\le\varepsilon/2$ 
and  $d\varphi_n\circ g^{\omega_n}\tilde{w}_n=\tilde{w}_n$ 
for some $\omega_n\in ]t_n-\varepsilon,t_n+\varepsilon[$. 
Here is a key geometric point:
the fact that on $T^1 M$, the orbits $g^sw_n$ and $g^sv_n$ are $\varepsilon$-close 
during the time $t_n$ is
expressed by the fact that the isometry $\varphi_n$ which sends
$g^{\omega_n}\tilde{w}_n$ on $\tilde{w}_n$ is the same which sends
$g^{t_n}\tilde{v}_n$ close to~$\tilde{v}_n$.

\quad
Indeed, as $\omega_n=t_n\pm\varepsilon$ and $\varphi_n$ is an isometry, the distance between
$g^{\omega_n}\tilde{w}_n$ and $g^{t_n}\tilde{v}_n$ is very close from the distance
between $\tilde{w_n}$ and $\tilde{v}_n$, which is  very small. 
By convexity of the riemannian distance on the Hadamard manifold $\tilde{M}$, 
we deduce that $\gamma_{\tilde{v_n}}(s)$
and $\gamma_{\tilde{w}_n}(s)$ stay close one another for $0\le s\le
\min(t_n,s_n)$. 
It also implies that $D(g^s \tilde{v}_n,g^s\tilde{w}_n)$ is small.  

\quad

Now, the idea of the proof is to show that for $n$ big enough, 
$\varphi_n$ is an axial isometry, and to
find on its axis a periodic vector $\tilde{w}_n$ 
converging to $v$, in contradiction
with the above assumption. 

\quad
Denote by $\gamma_{\tilde{v}}$  the geodesic determined by $\tilde{v}$. 
 Let $p$ (resp. $p_n$, $q_n$) be the base point of $\tilde{v}$
(resp. $\tilde{v}_n$, $g^{t_n}\tilde{v}_n$). 
As $\varphi_n^{-1}(p_n)$ is close to $\gamma_{\tilde{v_n}}(t_n)$, 
and $p_n\to p$, $v_n\to v$,
 it is easy to  check that 
$\varphi_n^{-1}(p)\to \gamma_{\tilde{v}}(+\infty)$. 
Similarly, one gets
 $\varphi_n(p)\to\gamma_{\tilde{v}}(-\infty)$.

We need the following lemma. 

\begin{lem}\cite[lemma 3.1]{ballmann}
 Let $\tilde{M}$ be a Hadamard rank one manifold,  $\tilde{v}\in T^1\tilde{M}$ a 
  regular vector, $x=\gamma_{\tilde{v}}(+\infty)$ and
  $y=\gamma_{\tilde{v}}(-\infty)$.
For all $\eta>0$, there exist neighbourhoods $V_{\eta}(x)$ and
  $V_{\eta}(y)$ of $x$ and $y$ respectively, such that for all $x'\in
  V_{\eta}(x)$, $y'\in V_{\eta}(y)$, there exists a geodesic
  $\gamma'$ joining $y$ to $x$ such that $d(\gamma_{\tilde{v}}(0),\gamma')\le
  \eta$. 
\end{lem}
This lemma is easy to prove on negatively curved manifolds with curvature
bounded by above by $-a^2<0$. 
A similar result is also true \cite{eb1} for regular vectors
on higher rank symmetric spaces. 
If $\tilde{v}$ is a singular vector on a rank one manifold, this lemma is
probably almost always false. However it does not mean that a closing lemma
could not be proven in another way for such vectors. \\

Let us now finish the proof of the closing lemma. 
Let $V_k(x)$ and $V_k(y)$ be the
neighbourhoods given by the above lemma for $\eta=\frac{1}{k}$.
 We can also assume that they are
homeomorphic to balls of $\R^{d-1}$ and their closures are homeomorphic to
closed balls of $\R^{d-1}$ (with $d=\dim \tilde{M}$). 
(Endowed with the cone topology, the boundary $\partial \tilde{M}$ is
homeomorphic to the $d-1$-dimensional sphere, \cite{eb-on}.)

\medskip

\quad 
Using the fact that $\varphi_n(p)\to y$ and $\varphi_n^{-1}(p)\to x$ one
proves easily that for $n$ big enough we have
$\varphi_n^{-1}(\overline{V_k(x)})\subset V_k(x)$ and
$\varphi_n(\overline{V_k(y)})\subset V_k(y)$. 
By the Brouwer fixed point theorem, we deduce
that $\varphi_n$  has two fixed points $x_n\in V_k(x)$ and $y_n\in V_k(y)$.
Consider a  geodesic joining $y_n$ to $x_n$ given by the above lemma. 
It is invariant by $\varphi_n$, 
which acts by translation on it. 
Let $w_n$ be the vector of this geodesic 
minimizing the distance to $\tilde{v}$, and $p'_n$ be its basepoint. 
This vector is close to
$\tilde{v}$, hence close to
$\tilde{v}_n$, 
if $n$ has been chosen big enough.

Denote by $\omega_{n}=d(p'_{n},\varphi_{n}(p'_{n}))$
the period associated to $w_n$. 
Since $p'_n$ is close to $p_n$, $\omega_n$
is close to $d(p_n,\varphi_n(p_n))$, and so close to 
$t_n$ for $n$ big enough.
Thus, we get the desired contradiction. 
\end{proof}





\end{document}